\documentclass{article}
\usepackage{fullpage}
\usepackage{cite}
\usepackage{amsmath,amssymb,amsfonts}
\usepackage{algorithmic}
\usepackage{amsthm}
\usepackage{graphicx}
\usepackage{textcomp}
\usepackage{lipsum}
    \usepackage{subfigure}
    \usepackage{float}
\newtheorem{theorem}{Theorem}
\newtheorem{lemma}{Lemma}
\newtheorem{remark}{Remark}
\newtheorem{example}{Example}
\DeclareMathOperator{\diag}{diag}

\DeclareMathOperator{\Diag}{Diag}

\newtheorem{myassump}{Assumption}

\def\BibTeX{{\rm B\kern-.05em{\sc i\kern-.025em b}\kern-.08em
    T\kern-.1667em\lower.7ex\hbox{E}\kern-.125emX}}

\newcommand\blfootnote[1]{%
  \begingroup
  \renewcommand\thefootnote{}\footnote{#1}%
  \addtocounter{footnote}{-1}%
  \endgroup
}

\begin{document}
	
	\begin{center}
		{\bf {\LARGE{Distributed gradient methods under heavy-tailed communication noise}}}
		
		\vspace*{.2in}
		
				\large{
					\begin{tabular}{cc}
						 Manojlo Vukovic$^\circ$, & Dusan Jakovetic$^\ast$,
					\end{tabular}
					\begin{tabular}{ccc}
						Dragana Bajovic$^\dagger$, & Soummya Kar$^\rhd$
					\end{tabular}
				}
			
		\vspace*{.2in}
			
		\begin{tabular}{c}
        $^\circ$University of Novi Sad, Faculty of Technical Sciences, Department of\\ Fundamental Sciences (manojlo.vukovic@uns.ac.rs)\\
		$^\ast$ University of Novi Sad, Faculty of Sciences, Department of\\ Mathematics and Informatics (dusan.jakovetic@dmi.uns.ac.rs)
		\\ $^\dagger$ University of Novi Sad, Faculty of Technical Sciences, Department of \\Power, Electronic and
		Communication Engineering(dbajovic@uns.ac.rs) \\ 
		$^\rhd$ Department of Electrical and Computer Engineering, \\Carnegie Mellon University (soummyak@andrew.cmu.edu)
		\\
		\end{tabular}
				
				\vspace*{.2in}
				
				\today
				
				\vspace*{.2in}
		
	\end{center}
	
	\begin{abstract}
We consider a standard distributed optimization problem in which networked nodes collaboratively minimize the sum of their locally known convex costs. 
     For this setting, we address for the first time the fundamental problem of design and analysis of distributed methods to solve the above problem when inter-node communication is subject to \emph{heavy-tailed} noise. 
     Heavy-tailed noise is highly relevant and frequently arises in densely deployed wireless sensor and Internet of Things (IoT) networks. 
     Specifically, we design a distributed gradient-type method that features a carefully balanced mixed time-scale time-varying consensus and gradient contribution step sizes and a bounded nonlinear operator on the consensus update to limit the effect of heavy-tailed noise. Assuming heterogeneous strongly convex local costs with mutually different minimizers that are arbitrarily far apart, we show that the proposed method converges to a neighborhood of the network-wide problem solution in the mean squared error (MSE) sense, and we also characterize the corresponding convergence rate. We further show that the asymptotic MSE can be made arbitrarily small through consensus step-size tuning, possibly at the cost of slowing down the transient error decay. Numerical experiments corroborate our findings and demonstrate the resilience of the proposed method to heavy-tailed (and infinite variance) communication noise. They also show that existing distributed methods, designed for finite-communication-noise-variance settings, fail in the presence of infinite variance noise.\blfootnote{\textbf{Acknowledgement.} The work of D. Jakovetic and M. Vukovic is supported by the Science Fund of the Republic of Serbia, GRANT No 7359, Project title- LASCADO. 
     The work of D. Bajovic and M. Vukovic is supported by  the Ministry of Science, Technological Development and Innovation (Contract No. 451-03-137/2025-03/200156) and the Faculty of Technical Sciences, University of Novi Sad through project “Scientific and Artistic Research Work of Researchers in Teaching and Associate Positions at the Faculty of Technical Sciences, University of Novi Sad 2025” (No. 01-50/295). The work of D. Jakovetic is  supported by the Ministry of Science, Technological Development and Innovation of the Republic of Serbia (Grants No. 451-03-137/2025-03/ 200125 \& 451-03-136/2025-03/ 200125)}
      
\end{abstract}

\section{Introduction}
    We consider a standard distributed optimization problem where  networked nodes (agents) collaborate among each other in order to minimize a global objective function, while each agent has information only on its own local objective function and can cooperate only with its neighboring agents. To be more precise, we consider standard distributed gradient strategies where each agent at each time step updates its current solution estimate by 1)  weight averaging its own and its neighbors’ solution estimates, and 2) moving it in the negative direction of its local objective gradient~\cite{NedicOzdaglar,dusanrate}.

    Distributed optimization has been extensively studied, e.g., ~\cite{dusanrate,nediccomnoise,shuhua,PuNedic,ShorinwaQuasi,DualAveraging,DualAveraging2,chatzipanagiotis2015augmented,augmented2,boyd2011distributed,adam2,convexmachine,ozdaglar}, and different types of methods have been proposed, including distributed gradient-like~\cite{nediccomnoise}, gradient-tracking~\cite{PuNedic}, quasi-Newton~\cite{ShorinwaQuasi} and Newton-like~\cite{liu2023communication}, dual averaging~\cite{DualAveraging,DualAveraging2}, and augmented Lagrangian~\cite{chatzipanagiotis2015augmented,augmented2} and alternating multiplier direction methods~\cite{boyd2011distributed,adam2}. The methods have been analyzed for different function classes and under various imperfections in computations (e.g., stochastic or deterministic gradient errors)~\cite{dusanrate,nediccomnoise} and communications (e.g., network failing links or noisy communication links)\cite{nediccomnoise}. 

    In this paper, we focus on the fundamental problem of distributed optimization in the presence of \emph{heavy-tailed communication noise}. In particular, unlike existing work on distributed optimization under communication noise~\cite{nediccomnoise,wangnoisysharing}, we allow the noise to have infinite variance. Heavy-tailed communication noise frequently arises in modern wireless communication systems such as densely deployed wireless sensor networks and Internet of Things (IoT) networks~\cite{IoTHeavyTail,IoTHeavyTail2,18,19,20}. For example, reference~\cite{win,Haenggi} theoretically shows that interference noise in a densely deployed wireless sensor network follows a stable distribution with heavy tails $\alpha$. On the other hand, reference~\cite{IoTHeavyTail} provides empirical evidence of heavy-tailed noise in practical IoT deployments.
    
 In more detail, we design a novel distributed gradient-type method that is provably resilient to additive heavy-tailed communication noise. The proposed method  incorporates a bounded nonlinearity in the consensus update part and utilizes carefully designed different time-scale weightings for the consensus and gradient update parts. The nonlinearity essentially  
    annihilates the effect of heavy tails by ``clipping'' large noise values and, as we show ahead, is necessary for handling the heavy-tailed communication noise. On the other hand, ``clipping'' in the consensus update also limits the information flow among the agents. We successfully address the noise suppression-information flow tradeoff via the careful two time-scale weighting of the consensus and gradient updates. 
    
    Assuming strongly convex local costs, We prove that the proposed method converges to a neighborhood of the global optimizer in the mean squared error (MSE)  sense, and we also characterize the rate of achieving the neighborhood. More precisely, we show that the MSE 
    is of the order $O(a^2+1/t^\gamma)$, 
    where $t$ is the iteration counter, $a/(t+1)$ is the gradient 
    weighting step-size and $\gamma>0.$ For the sufficiently small $a,$ it can been shown that $\gamma$ is linear function of $a,$ and therefore, there is a tradeoff between the limiting bound and the convergence rate with respect to the choice of $a.$
    Hence, the asymptotic error can be controlled by the gradient part step-size design. Numerical experiments on synthetic and real data confirm the analytical findings. In addition, we show by simulation that existing distributed optimization methods tailored for communication noise settings~\cite{nediccomnoise} break down in the presence of heavy-tailed noise.  

    We next briefly review the literature to help us contrast our paper with existing work. The methods in~\cite{dusanrate,nediccomnoise,shuhua} provide various convergence results for the methods proposed therein under different assumptions on model imperfections such as gradient noise and communication noise. For example, the authors in~\cite{dusanrate} establish a MSE convergence rate in the presence of (finite-variance) gradient noise. On the other hand, reference~\cite{nediccomnoise} establishes almost sure (a.s.) convergence under noisy communication links. To be more precise, the authors in~\cite{nediccomnoise} assume that the  variance of the communication noise is bounded by a certain function of time. 
    Moreover, noisy communications in distributed optimization have been analyzed in~\cite{Wang,wangnoisysharing}. However, these papers do not cover the case when communication between agents follows heavy- tailed distributions with infinite  variance. To the best of our knowledge, prior to this work, no convergence results in the presence of heavy tailed communication noise and infinite variance has been established in the distributed optimization context. 

    Heavy tails have been considered earlier in  centralized optimization~\cite{gardientHTnoise}. In distributed optimization, heavy tails have also been considered, but only in terms of gradient noise~\cite{shuhua}, not communication noise. Technically, introducing heavy-tailed communication noise leads to very different stochastic dynamics and requires different algorithmic design and analysis tools when compared with~\cite{shuhua}. Furthermore, heavy tailed noises have been studied in the context of distributed estimation~\cite{Ourwork,Ourwork1}. Reference~\cite{Ourwork,Ourwork1} establishes strong convergence results in the presence of heavy-tailed communication noise~\cite{Ourwork}, while reference~\cite{Ourwork1} additionally handles gradient (observation) noise. The distributed estimation problem studied in \cite{Ourwork,Ourwork1} may be viewed as a distributed optimization problem when local objective functions are quadratic and \emph{have a common optimizer}.  However. the general optimization problem considered here fundamentally lacks this property, i.e., the local optimizers are mutually different. In fact, this local solutions heterogeneity brings a major technical challenge here; thus a different algorithmic design and analysis tools are required when compared with \cite{Ourwork,Ourwork1}. Notably, as we show by simulation, a direct adaptation of the method in \cite{Ourwork} (with equal time-weighting of the consensus and gradient update parts) to the distributed optimization problem considered here leads to poor performance in general.

	\textbf{Paper organization}. In Section~\ref{section-model-algorithm} we describe the distributed optimization model that is considered and give all basic assumptions. Section~\ref{section-proposed-algorithm} presents the proposed nonlinear method. Section~\ref{section-convergenceanalysis} establishes convergence results. In Section~\ref{section:examples} we present numerical examples. The conclusion is given in Section~\ref{section-conlusion}.

	\textbf{Notation}. We denote by $\mathbb R$ the set of real numbers and by ${\mathbb R}^m$ the $m$-dimensional
	Euclidean real coordinate space. We use normal lower-case letters for scalars,
	lower case boldface letters for vectors, and upper case boldface letters for
	matrices. Further,  to represent a vector $\mathbf{a}\in\mathbb{R}^m$ through its component, we write $\mathbf{a}=[\mathbf{a}_1, \mathbf{a}_2, ...,\mathbf{a}_m]^\top$ and we denote by: $\mathbf{a}_i$ or $[\mathbf{a}_i]$, as appropriate, the $i$-th element of vector $\mathbf{a}$; $\mathbf{A}_{ij}$ or $[\mathbf{A}_{ij}]$, as appropriate, the entry in the $i$-th row and $j$-th column of
	a matrix $\mathbf{A}$;
	$\mathbf{A}^\top$ the transpose of a matrix $\mathbf{A}$; $\otimes$ the Kronecker product of matrices. Further, we use either  
	$\mathbf{a}^\top \mathbf{b}$ or 
	$\langle \mathbf{a},\,\mathbf{b}\rangle$ 
	for the inner products of vectors 
	$\mathbf{a}$ and $\mathbf{b}$. Next, we let  
	$\mathbf{I}$, $\mathbf{0}$, and $\mathbf{1}$ be, respectively, the identity matrix, the zero vector, and the column vector with unit entries; $\Diag(\mathbf{a})$ the diagonal matrix 
	whose diagonal entries are the elements of vector~$\mathbf{a}$; 
	$\mathbf{J}$ the $N \times N$ matrix $\mathbf{J}:=(1/N)\mathbf{1}\mathbf{1}^\top$.
	When appropriate, we indicate the matrix or vector dimension through a subscript.
	Next, $\mathbf{A}\succ  0 \,(\mathbf{A} \succeq  0 )$ means that
	the symmetric matrix $A$ is positive definite (respectively, positive semi-definite).
	We further denote by:
	$\|\cdot\|=\|\cdot\|_2$ the Euclidean (respectively, spectral) norm of its vector (respectively, matrix) argument; $\lambda_i(\cdot)$ the $i$-th smallest eigenvalue; $g^\prime(v)$ the derivative evaluated at $v$ of a function $g:\mathbb{R}\to\mathbb{R}$; $\nabla h(\mathbf{w})$ and $\nabla^2 h(\mathbf{w})$ the gradient and Hessian, respectively, evaluated at $w$ of a function $h: {\mathbb R}^m \rightarrow {\mathbb R}$, $m > 1$; $\mathbb P(\mathcal A)$ and $\mathbb E[u]$ the probability of
	an event $\mathcal A$ and expectation of a random variable $u$, respectively.
	Finally, for two positive sequences $\eta_n$ and $\chi_n$, we have: $\eta_n = O(\chi_n)$ if
	$\limsup_{n \rightarrow \infty}\frac{\eta_n}{\chi_n}<\infty$.

	\section{Problem model and assumptions}
	\label{section-model-algorithm}

    We consider a network of $N$ nodes where nodes mutually collaborate in aim to solve the following unconstrained problem
    \begin{align}\label{eq:problemmodel}
        \min f(x)= \sum\limits_{i=1}^N f_i(\mathbf{x}),
    \end{align}
    where function $f_i:\mathbb{R}^M\to \mathbb{R}$ is only available to node $i,$ $i=1,2,...,N.$
    The underlying probability space is denoted by $(\Omega,\mathcal{F},\mathbb{P})$.
    The underlying network topology is modeled via a graph $G=(V,E),$ where $V=\{1,...,N\}$ is the set of agents and $E$ is the set of links, i.e., $\{i,j\}\in E$ if there exists a link between agents $i$ and $j$. Moreover, we also define the set of all arcs $E_d$ in the following way: if $\{i,j\} \in E$ then $(i,j)\in E_d$ and $(j,i) \in E_d,$ where $(i,j)$ corresponds to communication froj $j$ to $i.$ The set of neighbors of node $i$ is denoted by  $\Omega_i=\{j\in V: \{i,j\}\in E\}$, the degree matrix is given by $\mathbf{D}=\Diag(\{d_i\})$, where $d_i=|\Omega_i|$ is the number of neighbors of agent~$i$, the adjacency matrix $\mathbf{A}$ is a zero-one symmetric matrix with zero diagonal, such that, for $i\neq j$, $\mathbf{A}_{ij}=1$ if and only if $\{i,j\}\in E$ and the graph Laplacian matrix $\mathbf{L}$ is defined by $\mathbf{L}=\mathbf{D}-\mathbf{A}.$ 
	\noindent We make the following assumptions.
    \begin{myassump}\label{as:newtworkmodel}\textbf{Network model:} \\
         \noindent Graph $G=(V,E)$ is connected, simple (no self or multiple links) and static.
	\end{myassump}
\noindent Notice that this assumption is equivalent to $\lambda_2(\mathbf{L})>0.$
    \begin{myassump}\label{as:solvability}\textbf{Solvability:} \\
        \noindent For all $i=1,2,...,N$ function $f_i:\mathbb{R}^M\to \mathbb{R}$ is twice continuously differentiable and 
            \begin{align*}
                \mu \mathbf{I} \preceq \nabla^2f_i(\mathbf{x} )\preceq L \mathbf{I}
            \end{align*}
            for all $\mathbf{x}\in\mathbb{R}^M,$ where $0<\mu\leq L<\infty.$ 
	\end{myassump}
\noindent This assumption implies that for each $i$ function $f_i$ is strongly convex and has Lipschitz continuous gradient, which ensures that problem~\eqref{eq:problemmodel} has the unique solution $\mathbf{x}^\star\in\mathbb{R}^M.$
	
\noindent For those nodes $i$ and $j$ for which there is an arc, we 
denote by $\boldsymbol{\xi}_{ij}^t$ communication noise that corrupts the signal when node $i$ is receiving it from node $j$ at time instant $t$ (see algorithm~\eqref{eq:alg1}).
	
	\begin{myassump}\label{as:communicationnoise}\textbf{Communication noise:}
		\begin{enumerate}
			\item Additive communication noise $\{\boldsymbol{\xi}^t_{ij}\}$, $\boldsymbol{\xi}^t_{ij} \in \mathbb{R}^M$ is i.i.d. in time $t$, and independent across different arcs~$(i,j)\in E_d.$
            \item Each random variable $[\boldsymbol{\xi}^t_{ij}]_{\ell}$, for 
			each $t=0,1...$, for each arc $(i,j)$, 
			for each entry $\ell=1,...,M$, 
			has the same probability density function~$p$.
            \item The pdf $p$ is symmetric, i.e. $p(u)=p(-u),$ for every $u\in\mathbb{R}$ and $p(u)>0$ for $|u|\leq c$, for some constant $c>0$;
            \item There holds that $\int |u|p(u)du <\infty$.
		\end{enumerate}
	\end{myassump}
    Notice here that the condition 3. in Assumption~\ref{as:communicationnoise} ensures that communication noise is zero mean. Also, notice that there is no assumption on the tails distribution nor on the variance of the communication noise. In particular, the variance can be infinite. Moreover, condition 2. in Assumption~\ref{as:communicationnoise} can be relaxed in the sense that entries $[\boldsymbol{\xi}^t_{ij}]_\ell$ and $[\boldsymbol{\xi}^t_{ij}]_s$ for $\ell\neq s$ can be mutually dependent.
	
	\section{Proposed algorithm}\label{section-proposed-algorithm}

    In order to reach a neighborhood of the solution $\mathbf{x}^\star$ of  problem~\eqref{eq:problemmodel} in the presence of heavy-tailed communication noise, we propose a distributed stochastic gradient strategy, where a nonlinearity in the communication update part is added in order to  annihilate the effect of the heavy-tailed noise. To be more precise, each node $i,$ $i=1,2,..,N$, at each time $t$, updates its solution estimate $\mathbf{x}_{i}^{t}$ by the following algorithm:
	
	\begin{align}\label{eq:alg1}
		\mathbf{x}_{i}^{t+1}=\mathbf{x}_{i}^{t}-\beta_{t}\sum_{j\in\Omega_{i}}
		\boldsymbol{\Psi}\left( \mathbf{x}_{i}^{t}-\mathbf{x}_{j}^{t} 
		+\boldsymbol{\xi}_{ij}^t\right)-\alpha_t\nabla f_i(\mathbf{x}_i).
	\end{align} %
    Here, $\alpha_t=\frac{a}{t+1}$ and $\beta_t=\frac{b}{(t+1)^\delta}$ are step sizes; and $a,b>0$ and $\frac{1}{2}<\delta<1$ are constants.
    As it will be clear later, the choice of constants $a, \delta$ is crucial in order for the algorithm to achieve a sufficiently small neighborhood of the solution $\mathbf{x}^\star$ (in the sense of MSE).  
    Moreover, the step size decay rate $\delta$ of $\beta_t,$ must be in $(1/2,1),$ since for other choices, algorithm~\eqref{eq:alg1} may fail to converge.   
    The map $\boldsymbol{\Psi}:  \mathbb{R}^M\to\mathbb{R}^M$ is a non-linear map;  $\boldsymbol{\Psi}$ operates component-wise, i.e., for
	 $\mathbf{x}\in\mathbb{R}^M$, we set that (slightly abusing notation) $\boldsymbol{\Psi}(\mathbf{x})=[{\Psi}(\mathbf{x}_1),{\Psi}(\mathbf{x}_2),...,{\Psi}(\mathbf{x}_M)].$
     We make the following assumption.
	 	
	\begin{myassump}\label{as:nonlinearity}\textbf{Nonlinearity $\Psi$:}\\
		The non-linear function $\Psi:\mathbb{R}\to\mathbb{R}$ satisfies the following properties:
		\begin{enumerate}
			\item Function~$\Psi$ is odd, i.e., $\Psi(a)=-\Psi(-a),$ for any $a\in\mathbb{R}$;
			\item $\Psi(a) > 0,$ for any $a > 0$.
			\item Function $\Psi$ is monotonically nondecreasing;
            \item $|\Psi(a)|\leq c_1$, for some constant $c_1>0.$
			\item $\Psi$ is continuous and differentiable for all $a\in\mathbb{R}$ and $|\Psi^\prime(a)|\leq c_2$ for some constant $c_2>0.$
		\end{enumerate}
	\end{myassump}
    There are many functions that satisfy Assumption~\ref{as:nonlinearity}, such as $\Psi_1(a)=\arctan(a), \Psi_2(a)=\frac{x}{\sqrt{1+x^2}}.$ 
	\noindent The purpose of nonlinearity~$\Psi$ is to mitigate the impact of the heavy tailed communication noise. 
    Assumption~\ref{as:nonlinearity} differs from the analogous assumption on the nonlinearity in~\cite{Ourwork,Ourwork1}, i.e., here we additionally assume that the nonlinearity is differentiable for all $a\in\mathbb{R}.$ The reasoning behind this is the necessity of linearization of the consensus part. (See ahead the proofs of Lemmas~\ref{lemma-MSE} and~\ref{lemma:disagreementbounds} and Theorem~\ref{theorem-MSE})

	\section{Convergence analysis}\label{section-convergenceanalysis}
	The analysis outline and strategy is as follows. 
	In subsection~\ref{subsection:PreliminariesandAuxiliaryResults}, we analyze algorithm~\eqref{eq:alg1} and write it in a more suitable form that will be used in the following proofs in subsection~\ref{subsection:MSE}. In  Subection~\ref{subsection:Boundedness}, we show that the quantity $\mathbf{x}^t$ generated by~\eqref{eq:alg1} is a.s. bounded by a linear function of $t^{1-\delta}.$ We also show that the MSE of the proposed algorithm~\eqref{eq:alg1} is uniformly bounded.
    Finally, in subsection~\ref{subsection:MSE}, we show that the proposed algorithm~\eqref{eq:alg1} achieves a solution neighborhood that depends on parameters $a$ and $\delta$ of the algorithm in the MSE sense. We also establish a MSE rate of convergence to the neighborhood.
	
	\subsection{Preliminaries and Auxiliary Results}
	\label{subsection:PreliminariesandAuxiliaryResults}

    First, it can be seen that algorithm~\eqref{eq:alg1} can be written in compact form, i.e., at each time $t=0,1,...$, algorithm~\eqref{eq:alg1} can be written as
	\begin{align}\label{eq:algcomp}
		\mathbf{x}^{t+1}=\mathbf{x}^{t}-\beta_t \mathbf{L}_{\boldsymbol{\Psi}}(\mathbf{x})-\alpha_t \nabla F(\mathbf{x^t}).
	\end{align}
	Here, $\mathbf{x}^t=[\mathbf{x}_1^t,\mathbf{x}_2^t,...,\mathbf{x}_N^t]^\top\in\mathbb{R}^{MN},$ map $\mathbf{L}_{\boldsymbol{\Psi}}(\mathbf{x}):\mathbb{R}^{MN}\to\mathbb{R}^{MN}$ is defined by
	\begin{align*}
		\mathbf{L}_{\boldsymbol{\Psi}}(\mathbf{x})=\begin{bmatrix}
			\vdots\\
			\sum\limits_{j\in\Omega_{i}}\boldsymbol{\Psi}(\mathbf{x}_i-\mathbf{x}_j+\boldsymbol{\xi}_{ij})\\
			\vdots
		\end{bmatrix},
	\end{align*}
	where the blocks $\sum\limits_{j\in\Omega_{i}}\boldsymbol{\Psi}(\mathbf{x}_i-\mathbf{x}_j+\boldsymbol{\xi}_{ij})\in\mathbb{R}^M$ are stacked one on top of another for $i=1,...,N,$ and function $F(\mathbf{x})=\sum\limits_{i=1}^N f_i(\mathbf{x}_i).$
    
	\noindent Next, we consider function $\varphi:\mathbb{R}\to\mathbb{R}$ that is given by
	\begin{align}
		\label{eq:phi-def}
		\varphi(a)=\int \Psi(a+w)p(w) dw,
	\end{align}
    where $\Psi:\mathbb{R}\to\mathbb{R}$ is a nonlinear function that satisfies Assumption~\ref{as:nonlinearity}, and $p$ is a probability density function of communication noise that satisfies Assumption~\ref{as:communicationnoise}. 
    We make use of the following Lemma from~\cite{PolyakNL}, see also~\cite{Ourwork,Ourwork1}.
	\begin{lemma} [Polyak Lemma~\cite{PolyakNL}]
		\label{Lemma-Polyak}
		Consider function $\varphi$ in \eqref{eq:phi-def}, where function $\Psi:\mathbb{R}\to\mathbb{R}$, satisfies Assumption \ref{as:nonlinearity}. Then, the following holds:
		\begin{enumerate}
			\item ${\varphi}$ is odd;
			\item If $|\Psi(\nu)| \leq c_1,$ for any $\nu \in \mathbb R$, then $|{\varphi}(a)| \leq c_1^{\prime},$ for any $a \in \mathbb R$, for some
			$c_1^\prime>0$;
			\item ${\varphi}(a)$ is monotonically nondecreasing;
			\item ${\varphi}(a) > 0,$ for any $ a>0$.
			\item ${\varphi}$ is continuous and differentiable for all $a\in\mathbb{R}$ with a strictly positive derivative at zero.
		\end{enumerate}
	\end{lemma}
    \begin{proof}
        For the proof of parts 1-4 see~\cite{PolyakNL}. Note that by boundedness of $\Psi'$ (see Assumption~\ref{as:nonlinearity}), and by using Lebesgue dominant theorem, differentiability of function $\varphi$ follows directly. 
    \end{proof} 
    Results from Lemma~\ref{Lemma-Polyak} ensure that function $\varphi$ has all key properties of function $\Psi$. Moreover, next lemma shows that $\varphi(a)$ can be bounded away from zero, since we have that $\varphi^\prime(0)>0.$
    \begin{lemma}\label{lemma:bound}
        Consider the function $\varphi$ in \eqref{eq:phi-def}. There exists a positive constant $G$ such that $|\varphi(a)|\geq \frac{\varphi'(0)\, G |a|}{2\,g},$ for all $0<|a|<g.$ Moreover, $\varphi'(c)\geq \frac{\varphi'(0)\, G}{2\,g},$ for $c\in (\min\{0,a\},\max\{0,a\})$ such that $\varphi(a)=\varphi'(c)a$ provided that $0<|a|<g.$
    \end{lemma}
    \begin{proof}
        Here, we only show that $\varphi'(c)\geq \frac{\varphi'(0)\, G}{2\,g},$ for $c\in \textrm{int}(0,a)$, see~\cite{gardientHTnoise} for the proof of the first part. Without loss of generality, let us assume that $a>0.$ By the mean value theorem we have that there exists a constant $c\in(0,a)$ such that $\varphi(a)=\varphi'(c) a.$ Therefore, we have that 
        \begin{align*}
            \varphi'(c) a=\varphi(a)\geq \frac{\varphi'(0)\, G a}{2\,g}
        \end{align*}
        and hence, $\varphi'(c)\geq \frac{\varphi'(0)\, G a}{2\,g}.$
    \end{proof}
    
    \noindent We also define function
    $\boldsymbol{\varphi}:\mathbb{R}^{M}\to\mathbb{R}^{M}$ as $\boldsymbol{\varphi}(\mathbf{x})=[\varphi(\mathbf{x}_1),\varphi(\mathbf{x}_2),...,\varphi(\mathbf{x}_{N})]$, where $\mathbf{x}\in\mathbb{R}^{N}$ and function $\varphi$ is the transformation defined by~\eqref{eq:phi-def} that corresponds to $\Psi.$
    Therefore, algorithm~\eqref{eq:algcomp} can be written as
    \begin{align}\label{eq:algfin}
		\mathbf{x}^{t+1}=\mathbf{x}^{t}-\beta_t\mathbf{L}_{\boldsymbol{\varphi}}(\mathbf{x}^t) -\alpha_t \nabla F(\mathbf{x^t}) - \beta_t \boldsymbol{\eta}^t.
	\end{align}
    Here, $\boldsymbol{\eta}^t\in\mathbb{R}^{MN}$ is given by
	\begin{align}\label{eq:eta}
		\boldsymbol{\eta}^t=\begin{bmatrix}
			\vdots\\
			\sum\limits_{j\in\Omega_{i}}\boldsymbol{\eta}_{ij}^t\\
			\vdots
		\end{bmatrix},
	\end{align}
	where $\boldsymbol{\eta}_{ij}^t=\boldsymbol{\Psi}(\mathbf{x}_i^t-\mathbf{x}_j^t+\boldsymbol{\xi}_{ij}^t)-\boldsymbol{\varphi}(\mathbf{x}_i^t-\mathbf{x}_j^t),$ 
	and $\mathbf{L}_{\boldsymbol{\varphi}}:\mathbb{R}^{MN}\to\mathbb{R}^{MN}$ as $\mathbf{L}_{\boldsymbol{\varphi}}(\cdot)=\mathbf{L}_{\boldsymbol{\Psi}}(\cdot)-\boldsymbol{\eta}^t,$ i.e., its $i$-th block of size $M$ is $\sum\limits_{j\in\Omega_{i}}\boldsymbol{\varphi}(\mathbf{x}_i-\mathbf{x}_j).$ for $i=1,2,...,N.$ Notice here that taking the expectation of $\boldsymbol{\eta}^t$ we have that $\mathbb{E}[\boldsymbol{\eta}^t]=0$ (for more detail see for example~\cite{Ourwork1}).
    
    \noindent Next, we state the following lemma from~\cite{gardientHTnoise} (see Lemma 8.1 in Appendix A).
    
    \begin{lemma} \label{lemma:boundedsequence} Consider a (deterministic) sequence 
    \begin{align*}
        v^{t+1}=\left(1-\frac{p}{(t+1)^\delta}\right)v^t+\frac{q}{(t+1)^\delta}, t\geq t_0
    \end{align*}
    with $p,q>0$ and $0<\delta\leq1,$ $t_0>0$ and $v^{t_0}\geq0.$ Further, assume that $t_0$ is such that $\frac{p}{(t+1)^\delta}\leq1,$ for all $t\geq t_0$. Then, $\lim\limits_{t\to\infty} v^t=\frac{q}{p}.$
    \end{lemma}

    \noindent Finally, we finish this subsection with the following lemma.

    \begin{lemma}\label{lemma:boundingsequence}
        Let $\{u_t\}, \{v_t\}$ and $\{r_t\}$ be (deterministic) sequences and $t_0>0$ such that 
        \begin{align}
            u_{t+1}&\leq \left(1- \frac{p}{t+1}\right)u_t+ \frac{q}{t+1},\label{eqn:ineqxy}\\
            v_{t+1}&=\left(1- \frac{p}{t+1}\right)v_t+ \frac{q}{t+1},\label{eqn:ydef}\\
            r_{t}&=v_t-\frac{q}{p}\label{eqn:zdef},
        \end{align}
        for $t\geq t_0,$ and some constants $p,q>0.$ If $u_{t_0}\leq v_{t_0},$ then for all $t\geq t_0,$ $u_t\leq v_t=r_t+\frac{q}{p}$, and $r_t= O(\frac{1}{t^p}).$
    \end{lemma}
    \begin{proof}
        By assuming that $u_t\leq v_t$ for some $t> t_0$ we have that
        \begin{align*}
            u_{t+1}&\leq \left(1- \frac{p}{t+1}\right)u_t+ \frac{q}{t+1}\\&\leq \left(1- \frac{p}{t+1}\right)v_t+ \frac{q}{t+1}=v_{t+1},
        \end{align*}
        and by initial assumption that $u_{t_0}\leq v_{t_0},$ we have that $u_t\leq v_t$ for all $t\geq t_0.$ Next, we have that
        \begin{align*}
            r_{t+1}&=v_{t+1}-\frac{q}{p}=\left(1- \frac{p}{t+1}\right)v_t+ \frac{q}{t+1}-\frac{q}{p}\\
            &=\left(1- \frac{p}{t+1}\right)\left( r_t+\frac{q}{p}\right)+ \frac{q}{t+1}-\frac{q}{p}\\
            &=\left(1- \frac{p}{t+1}\right) r_t.
        \end{align*}
        Hence,
        \begin{align*}
            r_t&=\left(1- \frac{p}{t}\right) r_{t-1}=\left(1- \frac{p}{t}\right)\left(1- \frac{p}{t-1}\right) r_{t-2}\\
            &=\cdots=r_{t_0}\prod\limits_{i=t_0+1}^t\left(1- \frac{p}{i}\right)\leq r_{t_0}\exp\left(-p \sum\limits_{i=t_0+1}^t\frac{1}{i}\right)\\
            &\leq  r_{t_0}\exp\left(-p \sum\limits_{i=t_0+1}^t\frac{1}{i}\right)\\
            &\leq r_{t_0}\exp\left(-p(\ln t -\ln t_0)\right)= r_{t_0} \left(\frac{t_0}{t}\right)^p,
        \end{align*}
        where we used that $1-x\leq \exp(-x),$ for all $x\in\mathbb{R}.$
    \end{proof}

	\subsection{Boundedness}\label{subsection:Boundedness}
    In this subsection, we first show that with~\eqref{eq:alg1} it holds that $\|\mathbf{x}^t\|$  is a.s. no greater than $c' + c''t^{1-\delta},$ for some constants $c',c''>0.$ Secondly, we show that $\mathbb{E}[\|\mathbf{x}^t \|^2]$ is uniformly bounded. 
    We have the following lemmas.

    \begin{lemma}\label{lemma:boundedx}
    Let Assumptions~\ref{as:newtworkmodel}-\ref{as:nonlinearity} hold. Then, we have that
    \begin{align}
        \|\mathbf{x}^t\|\leq M_t:=\| \mathbf{x}^{0}-\boldsymbol{\theta}^\star\| +2c_1^\prime\;b\;\frac{t^{1-\delta}}{1-\delta}, \quad \textit{a.s.},
    \end{align}
    provided that the step-sizes sequences $\{\alpha_t\}$ and $\{\beta_t\}$ are given by $\alpha_t=a/(t+1), \beta_t=b/(t+1)^\delta$ for $a,b>0,$ and $ \delta\in(0.5,1)$.
    \end{lemma}

    \begin{proof}
    Let us consider~\eqref{eq:algcomp} and by denoting $\boldsymbol{\theta}^\star=\mathbf{1}\otimes\mathbf{x}^\star,$ we get
     \begin{align*}
         \mathbf{x}^{t+1}-\boldsymbol{\theta}^\star&=\mathbf{x}^{t}-\boldsymbol{\theta}^\star- \alpha_t \left(\nabla F(\mathbf{x^t})-\nabla F(\boldsymbol{\theta}^\star) \right)\\ & - \alpha_t \nabla F(\boldsymbol{\theta}^\star)-\beta_t\mathbf{L}_{\boldsymbol{\Psi}}(\mathbf{x}^t).
    \end{align*}
    Notice here, that by mean value theorem we have that
    \begin{align}
        \nabla F(\mathbf{x^t})-\nabla F(\boldsymbol{\theta}^\star)&=\int\limits_0^1 \nabla^2 F(\boldsymbol{\theta}^\star + s(\mathbf{x}^t-\boldsymbol{\theta}^\star) ds (\mathbf{x}^t- \boldsymbol{\theta}^\star)\nonumber\\ &=\mathbf{H}^t (\mathbf{x}^t- \boldsymbol{\theta}^\star),\label{eqn:mvt}
    \end{align}
    where $\mu\mathbf{I}\preceq \mathbf{H}^t\preceq L \mathbf{I}.$ Thus,
    \begin{align*}
            \mathbf{x}^{t+1}-\boldsymbol{\theta}^\star&= \left(\mathbf{I} - \alpha_t \;\mathbf{H}^t\right)(\mathbf{x}^t- \boldsymbol{\theta}^\star) \\&- \beta_t\;\mathbf{L}_{\boldsymbol{\Psi}}(\mathbf{x}^t) - \alpha_t\; \nabla F(\boldsymbol{\theta}^\star).
    \end{align*}
    Therefore, 
    \begin{align*}
        \| \mathbf{x}^{t+1}-\boldsymbol{\theta}^\star\|&\leq \|\left(\mathbf{I} - \alpha_t \;\mathbf{H}^t\right)(\mathbf{x}^t- \boldsymbol{\theta}^\star)\| \\ &+ \beta_t \|\mathbf{L}_{\boldsymbol{\Psi}}(\mathbf{x}^t)\| + \alpha_t \|\nabla F(\boldsymbol{\theta}^\star)\|\\
        &\leq (1-\mu\alpha_t)\|\mathbf{x}^{t}-\boldsymbol{\theta}^\star\| + 2c_1^\prime \beta_t\\
        &\leq \|\mathbf{x}^{t}-\boldsymbol{\theta}^\star\| +2 c_1^\prime \beta_t
    \end{align*}
    for sufficiently large $t.$ Furthermore,
    \begin{align*}
        \| \mathbf{x}^{t}-\boldsymbol{\theta}^\star\| &\leq  \|\mathbf{x}^{t-1}-\boldsymbol{\theta}^\star\| +2 c_1^\prime \beta_{t-1} \\ &\leq \|\mathbf{x}^{t-2}-\boldsymbol{\theta}^\star\| + 2c_1^\prime \beta_{t-2}+2c_1^\prime\beta_{t-1}\\
        &\leq\cdots \leq \| \mathbf{x}^{0}-\boldsymbol{\theta}^\star\| + 2c_1^\prime \sum\limits_{s=0}^{t-1}\frac{b}{(s+1)^\delta}\\
        &\leq\| \mathbf{x}^{0}-\boldsymbol{\theta}^\star\| +2c_1^\prime\;b\;\frac{t^{1-\delta}}{1-\delta},
    \end{align*}
    which completes the proof.
    \end{proof}
    \begin{lemma}[Uniform Boundedness of MSE]\label{lemma-MSE} Let Assumptions \ref{as:newtworkmodel}-\ref{as:nonlinearity} hold. Then, for the sequence of iterates $\{\mathbf{x}^t\}$ 
		generated by algorithm~\eqref{eq:alg1}, provided that the step-sizes sequences $\{\alpha_t\}$ and $\{\beta_t\}$ are given by $\alpha_t=a/(t+1), \beta_t=b/(t+1)^\delta$ for $a,b>0,$ and $ \delta\in(0.5,1)$, there exits $t_0^\mathrm{mse}\geq0$ such that 
        $\mathbb{E}[\|\mathbf{x}^t-\mathbf{1}_{N}\otimes \mathbf{x}^\star \|^2]\leq \Tilde{c}=\max\left(\| \mathbf{x}^{0}-\boldsymbol{\theta}^\star\| +2c_1^\prime\;b\;\frac{(t_0^\mathrm{mse})^{1-\delta}}{1-\delta},\right.$ $\left.\frac{4\|\nabla F(\boldsymbol{\theta}^\star)\|^2}{\mu^2} +\mathbf{E}[\|\mathbf{x}^{t_0^\mathrm{mse}}-\boldsymbol{\theta}^\star\|^2]\right),$ for all $t=0,1,....$ 
	\end{lemma}

    \begin{proof}
    Let us consider~\eqref{eq:algfin} and recall that $\boldsymbol{\theta}^\star=\mathbf{1}\otimes\mathbf{x}^\star,$ we get
    \begin{align*}
         \mathbf{x}^{t+1}-\boldsymbol{\theta}^\star&=\mathbf{x}^{t}-\boldsymbol{\theta}^\star- \beta_t\mathbf{L}_{\boldsymbol{\varphi}}(\mathbf{x}^t) -\alpha_t \left(\nabla F(\mathbf{x^t})-\nabla F(\boldsymbol{\theta}^\star) \right) \\&- \beta_t \boldsymbol{\eta}^t - \alpha_t \nabla F(\boldsymbol{\theta}^\star).
    \end{align*}
     Since, function $\varphi$ is differentiable (see Lemma~\ref{Lemma-Polyak}), by mean value theorem, we have that there exist $M$ vectors $\mathbf{c}_{ij}\in\mathbb{R}^{N},$ such that 
    \begin{align}\label{eqn:laplaciantransformation}
    \mathbf{L}_{\boldsymbol{\varphi}}(\mathbf{x}^t) =\mathbf{L}_{\mathbf{c}} \mathbf{x}^t,
    \end{align}
    where $\mathbf{L}_{\mathbf{c}}\in\mathbb{R}^{NM\times NM}$ is block matrix such that $(i,j)$-th block is $\mathbf{D}_{ij}=\diag(\boldsymbol{\varphi}(\mathbf{c}_{ij}))$ for $(i,j)\in E^t, i\neq j,$ and $\mathbf{D}_{ii}=-\sum\limits_{j\in \Omega^t} \mathbf{D}_{ij}.$ Notice here, that matrix $\mathbf{L}_\mathbf{c}$ is semipositive definite matrix.
   Therefore, we have that 
    \begin{align*}
        \mathbf{x}^{t+1}-\boldsymbol{\theta}^\star&= \left(\mathbf{I}-\beta_t\; \mathbf{L}_{c}  - \alpha_t \;\mathbf{H}^t\right)(\mathbf{x}^t- \boldsymbol{\theta}^\star) \\&- \beta_t \;\boldsymbol{\eta}^t - \alpha_t\; \nabla F(\boldsymbol{\theta}^\star),
    \end{align*}
    where, $\mathbf{H}^t$ is defined by~\eqref{eqn:mvt}. Then, there holds
    \begin{align*}
        \|\mathbf{x}^{t+1}-\boldsymbol{\theta}^\star\|^2= \|\mathbf{m}^t\|^2-2\beta\; (\mathbf{m}^t)^\top\boldsymbol{\eta}^t+\beta_t^2\|\boldsymbol{\eta}^t\|^2,
    \end{align*}
    where $\mathbf{m}^t=\left(\mathbf{I}-\beta_t\; \mathbf{L}_{c}  - \alpha_t \;\mathbf{H}^t\right)(\mathbf{x}^t- \boldsymbol{\theta}^\star) - \alpha_t\; \nabla F(\boldsymbol{\theta}^\star).$ Taking expectation with respect to $\mathcal{F}_t$ we get
    \begin{align}\label{eqn:expectationwrtFt}
        \mathbf{E}[\|\mathbf{x}^{t+1}-\boldsymbol{\theta}^\star\|^2|\mathcal{F}_t]= \|\mathbf{m}^t\|^2 +\beta_t^2\;c_1^2.
    \end{align}
    Using Lemma~\ref{lemma:bound} and~\ref{lemma:boundedx} we get that
    \begin{align*}
        \|\mathbf{m}^t\| &\leq \|\left(\mathbf{I}-\beta_t\; \mathbf{L}_{c}  - \alpha_t \;\mathbf{H}^t\right)(\mathbf{x}^t- \boldsymbol{\theta}^\star)\|+\alpha_t\; \|\nabla F(\boldsymbol{\theta}^\star)\|\\
        &\leq \left(1-\mu\alpha_t\right)\|\mathbf{x}^{t}-\boldsymbol{\theta}^\star\|+\alpha_t\; \|\nabla F(\boldsymbol{\theta}^\star)\|,
    \end{align*}
    for sufficiently large $t.$ Using inequality $(s+q)^2\leq(1+\gamma)s^2+(1+1/\gamma)q^2,$ we have that 
    \begin{align*}
          \|\mathbf{m}^t\|^2&\leq \left(1+\mu\alpha_t\right)\left(1-\mu\alpha_t\right)^2\|\mathbf{x}^{t}-\boldsymbol{\theta}^\star\|^2\\ &+ (1+\frac{1}{\mu\alpha_t})\; \alpha_t^2\; \|\nabla F(\boldsymbol{\theta}^\star)\|^2\\
          &\leq \left(1-\mu\alpha_t\right)\|\mathbf{x}^{t}-\boldsymbol{\theta}^\star\|^2+\frac{\alpha_t}{\mu}\;\|\nabla F(\boldsymbol{\theta}^\star)\|^2
    \end{align*}
    Taking expectation in~\eqref{eqn:expectationwrtFt} we have that
    \begin{align}
         \mathbf{E}[\|\mathbf{x}^{t+1}-\boldsymbol{\theta}^\star\|^2]&\leq(1-\mu\alpha_t)\mathbf{E}[\|\mathbf{x}^{t}-\boldsymbol{\theta}^\star\|^2]\nonumber\\&+\frac{\alpha_t}{\mu}\; \|\nabla F(\boldsymbol{\theta}^\star)\|^2+\beta_t^2\;c_1^2\nonumber\\
         & \leq(1-\mu\alpha_t)\mathbf{E}[\|\mathbf{x}^{t}-\boldsymbol{\theta}^\star\|^2]\nonumber\\&+2\frac{\alpha_t}{\mu}\; \|\nabla F(\boldsymbol{\theta}^\star)\|^2.\label{eqn:uvrinequality1}
    \end{align}

    Utilizing Lemma~\ref{lemma:boundingsequence}, choosing $t_0^\mathrm{mse}$ such that~\eqref{eqn:uvrinequality1} hold for all $t\geq t_0^\mathrm{mse}$, and by setting that $u_t=\mathbf{E}[\|\mathbf{x}^{t}-\boldsymbol{\theta}^\star\|^2],$ $v_{t+1}=(1-\mu\alpha_t)v_t+2\frac{\alpha_t}{\mu}\; \|\nabla F(\boldsymbol{\theta}^\star)\|^2,$ $v_{t_0^\mathrm{mse}}=\mathbf{E}[\|\mathbf{x}^{t_0^\mathrm{mse}}-\boldsymbol{\theta}^\star\|^2],$ $r_t=v_t-\frac{2\|\nabla F(\boldsymbol{\theta}^\star)\|^2}{\mu^2}$
   for $t\geq t_0^\mathrm{mse},$ we get that 
   \begin{align*}
       \mathbf{E}[\|\mathbf{x}^{t+1}-\boldsymbol{\theta}^\star\|^2]&\leq \frac{2\|\nabla F(\boldsymbol{\theta}^\star)\|^2}{\mu^2}+|r_{t_0^\mathrm{mse}}|\left(\frac{t_0^\mathrm{mse}}{t}\right)^{\mu a}\\
       &=\frac{2\|\nabla F(\boldsymbol{\theta}^\star)\|^2}{\mu^2}\\&+\left|\mathbf{E}[\|\mathbf{x}^{t_0^\mathrm{mse}}-\boldsymbol{\theta}^\star\|^2]-\frac{2\|\nabla F(\boldsymbol{\theta}^\star)\|^2}{\mu^2}\right|\left(\frac{t_0^\mathrm{mse}}{t}\right)^{\mu a}\\
       &\leq \frac{2\|\nabla F(\boldsymbol{\theta}^\star)\|^2}{\mu^2}+\\&\left|\mathbf{E}[\|\mathbf{x}^{t_0^\mathrm{mse}}-\boldsymbol{\theta}^\star\|^2]-\frac{2\|\nabla F(\boldsymbol{\theta}^\star)\|^2}{\mu^2}\right|\\
       &\leq \frac{4\|\nabla F(\boldsymbol{\theta}^\star)\|^2}{\mu^2} +\mathbf{E}[\|\mathbf{x}^{t_0^\mathrm{mse}}-\boldsymbol{\theta}^\star\|^2]
   \end{align*}
   since $t\geq t_0^\mathrm{mse}.$ On the other hand, for $t=1,2,...,t_0^\mathrm{mse}$ from Lemma~\ref{lemma:boundedx}, we have that 
   \begin{align*}
       \mathbf{E}[\|\mathbf{x}^{t}-\boldsymbol{\theta}^\star\|^2]\leq \| \mathbf{x}^{0}-\boldsymbol{\theta}^\star\| +2c_1^\prime\;b\;\frac{(t_0^\mathrm{mse})^{1-\delta}}{1-\delta}
   \end{align*}
   
   Therefore, we have that $\mathbf{E}[\|\mathbf{x}^{t}-\boldsymbol{\theta}^\star\|^2] \leq \Tilde{c},$
   for 
   \begin{align*}
   \Tilde{c}=&\max\left(\| \mathbf{x}^{0}-\boldsymbol{\theta}^\star\| +2c_1^\prime\;b\;\frac{(t_0^\mathrm{mse})^{1-\delta}}{1-\delta},\right.\\&\left.\frac{4\|\nabla F(\boldsymbol{\theta}^\star)\|^2}{\mu^2} +\mathbf{E}[\|\mathbf{x}^{t_0^\mathrm{mse}}-\boldsymbol{\theta}^\star\|^2]\right).
   \end{align*}
    \end{proof}
    \begin{remark}
        Notice that Lemma~\ref{lemma-MSE} implies uniform mean square boundedness of gradients computed at the iterates of algorithm~\eqref{eq:alg1}. Namely, since $\nabla F$ is Lipschitz continuous, we have that $\mathbb{E}[\|\nabla F(\mathbf{x}_t)\|^2]= \mathbb{E}[\|\nabla F(\mathbf{x}_t) - \nabla F(\boldsymbol{\theta}^\star) + \nabla F(\boldsymbol{\theta}^\star)\|^2] \leq 2 \mathbb{E}[\|\nabla F(\mathbf{x}_t) - \nabla F(\boldsymbol{\theta}^\star)\|^2] +2 \mathbb{E}[\|\nabla F(\boldsymbol{\theta}^\star)\|^2]\leq 2L \mathbb{E}[\|\mathbf{x}^{t}-\boldsymbol{\theta}^\star\|^2] +  2\|\nabla F(\boldsymbol{\theta}^\star)\|^2\leq 2L\Tilde{c}+2\|\nabla F(\boldsymbol{\theta}^\star)\|^2.$
    \end{remark}
    
	\subsection{Main result}\label{subsection:MSE}
    In this subsection, we state and prove the main result on the MSE performance of algorithm~\eqref{eq:alg1}. 
    Recall quantities $c^{\prime}$ from Lemma~\ref{Lemma-Polyak} and $G$ from Lemma~\ref{lemma:bound}.
    We have the following theorem.
    \begin{theorem}[MSE convergence rate]\label{theorem-MSE} Let Assumptions \ref{as:newtworkmodel}-\ref{as:nonlinearity} hold. Then, for the sequence of iterates $\{\mathbf{x}^t\}$ 
		generated by algorithm~\eqref{eq:alg1}, provided that the step-sizes sequences $\{\alpha_t\}$ and $\{\beta_t\}$ are given by $\alpha_t=a/(t+1), \beta_t=b/(t+1)^\delta$ for $a,b>0,$ and $ \delta\in(0.5,1)$, we have that 
        \begin{align*}
            &\limsup\limits_{t\to\infty}\mathbb{E}[\|\mathbf{x}^t-\mathbf{1}_{N}\otimes \mathbf{x}^\star \|^2]\\&\leq\frac{256a^2{c_1^\prime}^2(MN-1)\left( L\Tilde{c}+\|\nabla F(\boldsymbol{\theta}^\star)\|^2 \right)}{((1-
    \delta)\varphi'(0)G)^2}\left(1+4\frac{L}{\mu}\right).
        \end{align*}
        Moreover, 
        \begin{align*}
            &\mathbb{E}[\|\mathbf{x}^t-\mathbf{1}_{N}\otimes \mathbf{x}^\star \|^2]\\&\leq\frac{256a^2{c_1^\prime}^2(MN-1)\left( L\Tilde{c}+\|\nabla F(\boldsymbol{\theta}^\star)\|^2 \right)}{((1-
    \delta)\varphi'(0)G)^2}\left(1+4\frac{L}{\mu}\right)\\&+O\left(\frac{1}{t^{\gamma}}\right),
        \end{align*}
         for $\gamma=\min(\frac{\varphi'(0)G(1-\delta)}{4c_1^\prime},a\mu).$ 
	\end{theorem}
	Theorem~\ref{theorem-MSE} establishes 
	a bound on the MSE of the 
	proposed algorithm~\eqref{eq:algcomp} in the presence of heavy-tailed (possibly infinite variance) communication noise. The bound can be made arbitrarily small by setting parameter $a$ small enough. 
    In addition, we can see that, for $a < \frac{\varphi'(0)G(1-\delta)}{4\mu c_1^\prime} $, there is a tradeoff between the limiting accuracy and the convergence speed regarding the choice of a.
    Moreover, the bound can also be controlled by step size rate $\delta$ in $\beta_t,$ but $\delta$ can not be chosen such that bound is arbitrarily small, since $\delta\in(0.5,1).$ 
    Furthermore, regarding step size rates, extending MSE convergence results from distributed estimation with heavy-tailed communication noise~\cite{Ourwork} and have same step size rates in $\alpha_t$ and $\beta_t,$ leads to poor performance.
    On the other hand, it can be shown that existing methods~\cite{nediccomnoise} may fail to converge if communication between agents is corrupted by noise that has infinite variance (see Section~\ref{section:examples} ahead).
    
    To prove Theorem~\ref{theorem-MSE} we firstly consider the difference between the solution estimate of each agent with the (hypothetical) global average estimate $\overline{\mathbf{x}}^t=\frac{1}{N}\sum\limits_{i=1}^N \mathbf{x}_i^t.$ That is, we consider $\Tilde{\mathbf{x}}_i^t=\mathbf{x}_i^t-\overline{\mathbf{x}}^t,$ i.e., $\Tilde{\mathbf{x}}^t=(\mathbf{I}-\mathbf{J})\mathbf{x}^t,$ where $\mathbf{J}=\frac{1}{N}\mathbf{1}\mathbf{1}^\top\otimes \mathbf{I}.$ We have the following Lemma.
    
	\begin{lemma}\label{lemma:disagreementbounds}
    Consider algorithm~\eqref{eq:alg1}, and let Assumptions 1-4 hold. Then, there holds:
    \begin{align*}
       \limsup\limits_{t\to\infty} \mathbb{E}[\|\Tilde{\mathbf{x}}^t\|^2]\leq\frac{128a^2{c_1^\prime}^2(MN-1)\left( L\Tilde{c}+\|\nabla F(\boldsymbol{\theta}^\star)\|^2 \right)}{((1-
    \delta)\varphi'(0)G)^2}.
    \end{align*}
    Moreover, 
    \begin{align*}
        \mathbb{E}[\|\Tilde{\mathbf{x}}^t\|^2]&\leq \frac{128a^2{c_1^\prime}^2(MN-1)\left( L\Tilde{c}+\|\nabla F(\boldsymbol{\theta}^\star)\|^2 \right)}{((1-
    \delta)\varphi'(0)G)^2}\\&+O\left(t^{-\frac{\varphi'(0)G(1-\delta)}{4c_1^\prime}}\right).
    \end{align*}
	\end{lemma}
    
    \begin{proof}
     By definition of $\Tilde{\mathbf{x}}^t,$ we have that
    \begin{align*}
        \Tilde{\mathbf{x}}^{t+1}&=\Tilde{\mathbf{x}}^t-\beta_t \left(\mathbf{L}_{\boldsymbol{\varphi}}(\mathbf{x}^t)+\boldsymbol{\eta}^t\right) -\alpha_t \nabla F(\mathbf{x^t})\\&-\mathbf{J}\left(-\beta_t \left(\mathbf{L}_{\boldsymbol{\varphi}}(\mathbf{x}^t)+\boldsymbol{\eta}^t\right) -\alpha_t \nabla F(\mathbf{x^t}) \right)\\
        &=\Tilde{\mathbf{x}}^t- \beta_t \mathbf{L}_{\boldsymbol{\varphi}}(\Tilde{\mathbf{x}}^t)-\alpha_t(\mathbf{I}-\mathbf{J})\nabla F(\mathbf{x^t}) -\beta_t(\mathbf{I}-\mathbf{J})\boldsymbol{\eta}^t.
    \end{align*}
    Setting that $\mathbf{z}^t=\Tilde{\mathbf{x}}^t- \beta_t \mathbf{L}_{\boldsymbol{\varphi}}(\Tilde{\mathbf{x}}^t)-\alpha_t(\mathbf{I}-\mathbf{J})\nabla F(\mathbf{x^t}) ,$ we have that
    \begin{align*}
        \|\Tilde{\mathbf{x}}^{t+1}\|^2=\|\mathbf{z}^t\|^2-2\beta_t(\mathbf{z}^t)^\top (\mathbf{I}-\mathbf{J})\boldsymbol{\eta}^t + \beta_t^2\|(\mathbf{I}-\mathbf{J})\boldsymbol{\eta}^t\|^2,
    \end{align*}
    and consequently,
    \begin{align}\label{eqn:uptottilde}
        \mathbb{E}[\|\Tilde{\mathbf{x}}^{t+1}\|^2|\mathcal{F}_t]= \mathbb{E}[\|\mathbf{z}^t\|^2|\mathcal{F}_t]+\beta_t^2 c_2,
    \end{align}
    for some constant $c_2>0.$ Furthermore,
    \begin{align*}
        \|\mathbf{z}^t\|\leq\|\Tilde{\mathbf{x}}^t- \beta_t \mathbf{L}_{\boldsymbol{\varphi}}(\Tilde{\mathbf{x}}^t)\|+\alpha_t\|\underbrace{(\mathbf{I}-\mathbf{J})\nabla F(\mathbf{x^t}) }_{\boldsymbol{\zeta}_t}\|,
    \end{align*}
    and from~\eqref{eqn:laplaciantransformation} and utilizing Lemma~\ref{lemma:bound} and Lemma~\ref{lemma:boundedx} and the fact that  $\lambda_2(\mathbf{L})>0,$ we have that 
    \begin{align*}
        \|\mathbf{z}^t\|\leq\left(1-\beta_t\frac{\varphi'(0)\, G}{2\,M_t}\right)\|\Tilde{\mathbf{x}}^t\|+\alpha_t\|\boldsymbol{\zeta}_t\|.
    \end{align*}
    Again, by inequality $(s+q)^2\leq(1+\gamma)s^2+(1+1/\gamma)q^2,$ for $\gamma=\beta_t\frac{\varphi'(0)\, G}{2\,M_t},$ we have that 
    \begin{align*}
       \|\mathbf{z}^t\|^2&\leq\left(1+\beta_t\frac{\varphi'(0)\, G}{2\,M_t}\right)\left(1-\beta_t\frac{\varphi'(0)\, G}{2\,M_t}\right)^2\|\Tilde{\mathbf{x}}^t\|^2
        \\&+\left(1+\frac{2\,M_t}{\beta_t\varphi'(0)\, G}\right)\alpha_t^2 \|\boldsymbol{\zeta}_t\|^2\\
        &\leq \left(1-\beta_t\frac{\varphi'(0)\, G}{2\,M_t}\right)\|\Tilde{\mathbf{x}}^t\|^2+\left(1+\frac{2\,M_t}{\beta_t\varphi'(0)\, G}\right)\alpha_t^2\|\boldsymbol{\zeta}_t\|^2\\
        &\leq \left(1-\beta_t\frac{\varphi'(0)\, G}{2\,M_t}\right)\|\Tilde{\mathbf{x}}^t\|^2+\frac{4\,M_t \alpha_t^2}{\beta_t\varphi'(0)\, G}\|\boldsymbol{\zeta}_t\|^2.
    \end{align*}
    Finally, taking expectation in~\eqref{eqn:uptottilde}, we have
    \begin{align}
        \mathbb{E}[\|\Tilde{\mathbf{x}}^{t+1}\|^2]&\leq\left(1-\beta_t\frac{\varphi'(0)\, G}{2\,M_t}\right)\mathbb{E}[\|\Tilde{\mathbf{x}}^t\|^2]\\&+\frac{4\,M_t \alpha_t^2}{\beta_t\varphi'(0)\, G}\mathbb{E}[\|\boldsymbol{\zeta}_t\|^2]+\beta_t^2c_2\nonumber\\
        &\leq \left(1-\beta_t\frac{\varphi'(0)\, G}{2\,M_t}\right)\mathbb{E}[\|\Tilde{\mathbf{x}}^t\|^2]\\&+\frac{8\,M_t \alpha_t^2}{\beta_t\varphi'(0)\, G} 2(MN-1)\left( L\Tilde{c}+\|\nabla F(\boldsymbol{\theta}^\star)\|^2 \right)
        ,\label{eqn:uvrinequality2}
    \end{align}
    for $t$ large enough, since $\frac{M_t\alpha_t^2}{\beta_t}\sim \frac{c_3}{t+1}$ for $t$ large enough and some constant $c_3>0.$ Therefore, by Lemma~\ref{lemma:boundedsequence} we have that 
    \begin{align*}
        &\limsup\limits_{t\to\infty}\mathbb{E}[\|\Tilde{\mathbf{x}}^{t}\|^2] \\&\leq \lim\limits_{t\to\infty} \frac{\frac{8\,M_t \alpha_t^2}{\beta_t\varphi'(0)\, G}2(MN-1)\left( L\Tilde{c}+\|\nabla F(\boldsymbol{\theta}^\star)\|^2 \right)}{\beta_t\frac{\varphi'(0)\, G}{2\,M_t}}\\&=\lim\limits_{t\to\infty}\frac{32M_t^2\alpha_t^2(MN-1)\left( L\Tilde{c}+\|\nabla F(\boldsymbol{\theta}^\star)\|^2 \right)}{(\beta_t\varphi'(0)\, G)^2}\\&=\frac{128a^2{c_1^\prime}^2(MN-1)\left( L\Tilde{c}+\|\nabla F(\boldsymbol{\theta}^\star)\|^2 \right)}{((1-
    \delta)\varphi'(0)G)^2},
    \end{align*}
     which completes the first part of the proof. Next, utilizing Lemma~\ref{lemma:boundingsequence} and choosing $\Tilde{t}_0>0$ such that \eqref{eqn:uvrinequality2} hold for all $t\geq \Tilde{t}_0,$ we get that 
    \begin{align*}
    \mathbb{E}[\|\Tilde{\mathbf{x}}^{t}\|^2]&\leq \frac{128a^2{c_1^\prime}^2(MN-1)\left( L\Tilde{c}+\|\nabla F(\boldsymbol{\theta}^\star)\|^2 \right)}{((1-
    \delta)\varphi'(0)G)^2}\\&+O\left(t^{-\frac{\varphi'(0)G(1-\delta)}{4c_1^\prime}}\right)
    \end{align*}
     \end{proof}
    We are now ready to prove Theorem~\ref{theorem-MSE}
    
    \begin{proof}
        \textit{Proof of Theorem~\ref{theorem-MSE}}\\
        Recalling definition of global average  $\overline{\mathbf{x}}^t=\frac{1}{N}\sum\limits_{i=1}^N \mathbf{x}^t$ we have that
        \begin{align*}
         \overline{\mathbf{x}}^{t+1}&=\overline{\mathbf{x}}^t-\frac{\alpha_t}{N} \sum\limits_{i=1}^N \nabla f_i(\mathbf{x}_i^t)-\frac{\beta_t}{N}\sum\limits_{i=1}^N \boldsymbol{\eta}^t_i\\
        &=\overline{\mathbf{x}}^t-\frac{\alpha_t}{N} \sum\limits_{i=1}^N \left( \nabla f_i(\mathbf{x}_i^t) - \nabla f_i(\overline{\mathbf{x}}^t)+ \nabla f_i(\overline{\mathbf{x}}^t)\right)\\&-\beta_n \overline{\boldsymbol{\eta}}^t\\
        &=\overline{\mathbf{x}}^t-\frac{\alpha_t}{N} \sum\limits_{i=1}^N \nabla f_i(\overline{\mathbf{x}}^t)
        -\frac{\alpha_t}{N} \sum\limits_{i=1}^N \left( \nabla f_i(\mathbf{x}_i^t) - \nabla f_i(\overline{\mathbf{x}}^t)\right)\\&-\beta_n \overline{\boldsymbol{\eta}}^t\\
        &=\overline{\mathbf{x}}^t-\frac{\alpha_t}{N}\nabla f(\overline{\mathbf{x}}^t)
        -\frac{\alpha_t}{N} \sum\limits_{i=1}^N \left( \nabla f_i(\mathbf{x}_i^t) - \nabla f_i(\overline{\mathbf{x}}^t)\right)\\&-\beta_n \overline{\boldsymbol{\eta}}^t.
        \end{align*}
        Therefore,
        \begin{align*}
        \overline{\mathbf{x}}^{t+1}-\mathbf{x}^\star&=\overline{\mathbf{x}}^t-\mathbf{x}^\star-\frac{\alpha_t}{N}\left(\nabla f(\overline{\mathbf{x}}^t)-f(\mathbf{x}^\star\right)
        \\&-\frac{\alpha_t}{N} \sum\limits_{i=1}^N \left( \nabla f_i(\mathbf{x}_i^t) - \nabla f_i(\overline{\mathbf{x}}^t)\right)-\beta_n \overline{\boldsymbol{\eta}}^t,    
        \end{align*}
        and by following same idea as before, and setting that $\mathbf{s}_t=\overline{\mathbf{x}}^t-\mathbf{x}^\star-\frac{\alpha_t}{N}\left(\nabla f(\overline{\mathbf{x}}^t)-f(\mathbf{x}^\star\right)
        -\frac{\alpha_t}{N} \sum\limits_{i=1}^N \left( \nabla f_i(\mathbf{x}_i^t) - \nabla f_i(\overline{\mathbf{x}}^t)\right),$ we have that 
        \begin{align}\label{boundst}
            \mathbb{E}[\|\overline{\mathbf{x}}^{t+1}-\mathbf{x}^\star\|^2|\mathcal{F}_t]\leq \mathbb{E}[\|\mathbf{s}_t\|^2|\mathcal{F}_t]+\beta_t^2c_4.
        \end{align}
        for some constant $c_4>0.$ Next, we have that
        \begin{align*}
        \|\mathbf{s}_t\|&\leq \|\overline{\mathbf{x}}^t-\mathbf{x}^\star-\frac{\alpha_t}{N}\left(\nabla f(\overline{\mathbf{x}}^t)-f(\mathbf{x}^\star\right)\|\\&+\frac{\alpha_t}{N}\| \sum\limits_{i=1}^N \left( \nabla f_i(\mathbf{x}_i^t) - \nabla f_i(\overline{\mathbf{x}}^t)\right)\|.
        \end{align*}
        By mean value theorem we have that
        \begin{align*}
            \nabla f(\mathbf{\overline{x}^t})-\nabla f(\mathbf{x}^\star)&=\int\limits_0^1 \nabla^2 f(\mathbf{x}^\star + s(\overline{\mathbf{x}}^t-\mathbf{x}^\star) ds (\overline{\mathbf{x}}^t- \mathbf{x}^\star)\\&=\overline{\mathbf{H}}^t (\mathbf{x}^t- \mathbf{x}^\star),
        \end{align*}
        where $N\mu\mathbf{I}\preceq \overline{\mathbf{H}}^t\preceq NL \mathbf{I}.$ Moreover, we have that $\| \sum\limits_{i=1}^N \left( \nabla f_i(\mathbf{x}_i^t) - \nabla f_i(\overline{\mathbf{x}}^t)\right)\|\leq L \sum\limits_{i=1}^N \|\mathbf{x}_i^t-\overline{\mathbf{x}}^t\|.$
        Thus,
        \begin{align*}
             \|\mathbf{s}_t\|\leq (1 - \mu\alpha_t) \|\overline{\mathbf{x}}^t-\mathbf{x}^\star\|+\frac{\alpha_t}{N} L \sum\limits_{i=1}^N \|\mathbf{x}_i^t-\overline{\mathbf{x}}^t\|.
        \end{align*}
    Again, by inequality $(s+q)^2\leq(1+\gamma)s^2+(1+1/\gamma)q^2,$ for $\gamma=\mu\alpha_t,$ and by $(\sum\limits_{i=1}^N r_i)^2\leq N \sum\limits_{i=1}^N r_i^2,$ we have that 
    \begin{align*}
         \|\mathbf{s}_t\|^2&\leq (1 + \mu\alpha_t)(1 - \mu\alpha_t)^2 \|\overline{\mathbf{x}}^t-\mathbf{x}^\star\|^2\\&+\left(1 + \frac{1}{\mu\alpha_t}\right)\frac{\alpha_t^2}{N} L \sum\limits_{i=1}^N \|\mathbf{x}_i^t-\overline{\mathbf{x}}^t\|^2\\
         &\leq (1 - \mu\alpha_t) \|\overline{\mathbf{x}}^t-\mathbf{x}^\star\|^2+ \frac{2\alpha_t}{N} L \sum\limits_{i=1}^N \|\mathbf{x}_i^t-\overline{\mathbf{x}}^t\|^2.
    \end{align*}
    Taking expectation in~\eqref{boundst}, we get that
    \begin{align}
        \mathbb{E}[\|\overline{\mathbf{x}}^{t+1}-\mathbf{x}^\star\|^2]&\leq  (1 - \mu\alpha_t) \mathbb{E}[ \|\overline{\mathbf{x}}^t-\mathbf{x}^\star\|^2]\nonumber \\&+ \frac{2\alpha_t}{N} L \sum\limits_{i=1}^N \mathbb{E}[\|\mathbf{x}_i^t-\overline{\mathbf{x}}^t\|^2]+ \beta_t^2 c_4\nonumber\\
        &\leq  (1 - \mu\alpha_t) \mathbb{E}[ \|\overline{\mathbf{x}}^t-\mathbf{x}^\star\|^2] + \nonumber\\&\frac{4\alpha_t}{N} L \frac{128a^2{c_1^\prime}^2(MN-1)\left( L\Tilde{c}+\|\nabla F(\boldsymbol{\theta}^\star)\|^2 \right)}{((1-
    \delta)\varphi'(0)G)^2},\label{eqn:uvrinequality3}
    \end{align}
    and hence, by Lemma~\ref{lemma:boundedsequence} we have that 
    \begin{align*}
    &\limsup\limits_{t\to\infty}\mathbb{E}[\|\overline{\mathbf{x}}^{t}-\mathbf{x}^\star\|^2]\\&\leq\frac{\frac{4\alpha_t}{N} L \frac{128a^2{c_1^\prime}^2(MN-1)\left( L\Tilde{c}+\|\nabla F(\boldsymbol{\theta}^\star)\|^2 \right)}{((1-
    \delta)\varphi'(0)G)^2}}{\mu\alpha_t}\\
    &=\frac{512a^2{c_1^\prime}^2L(MN-1)\left( L\Tilde{c}+\|\nabla F(\boldsymbol{\theta}^\star)\|^2 \right)}{N\mu((1-\delta)\varphi'(0)G)^2}.
    \end{align*}
    To prove first part of Theorem~\ref{theorem-MSE}, we use that 
    \begin{align}\label{eqn:finale}
        \mathbb{E}[\|\mathbf{x}^{t}-\boldsymbol{\theta}^\star\|^2]\leq 2 \mathbb{E}[\|\Tilde{\mathbf{x}}^{t}\|^2] + 2N\mathbb{E}[\|\overline{\mathbf{x}}^{t}-\mathbf{x}^\star\|^2],
    \end{align}
    and get that
    \begin{align*}
         &\limsup\limits_{t\to\infty} \mathbb{E}[\|\mathbf{x}^{t}-\boldsymbol{\theta}^\star\|^2]\\&
       \leq 2 \frac{128a^2{c_1^\prime}^2(MN-1)\left( L\Tilde{c}+\|\nabla F(\boldsymbol{\theta}^\star)\|^2 \right)}{((1-
    \delta)\varphi'(0)G)^2} \\&+ 2 \frac{512a^2{c_1^\prime}^2L(MN-1)\left( L\Tilde{c}+\|\nabla F(\boldsymbol{\theta}^\star)\|^2 \right)}{\mu((1-\delta)\varphi'(0)G)^2} \\
    =&\frac{256a^2{c_1^\prime}^2(MN-1)\left( L\Tilde{c}+\|\nabla F(\boldsymbol{\theta}^\star)\|^2 \right)}{((1-
    \delta)\varphi'(0)G)^2}\left(1+4\frac{L}{\mu}\right)
    \end{align*}
     To prove second part of the theorem, again we utilize Lemma~\ref{lemma:boundingsequence} and choose $t_0$ such that~\eqref{eqn:uvrinequality3} hold for all $t\geq t_0,$ and get that
    \begin{align*}
        \mathbb{E}[\|\overline{\mathbf{x}}^{t}-\mathbf{x}^\star\|^2]&\leq \frac{512a^2{c_1^\prime}^2L(MN-1)\left( L\Tilde{c}+\|\nabla F(\boldsymbol{\theta}^\star)\|^2 \right)}{N\mu((1-\delta)\varphi'(0)G)^2}\\&+ O\left(t^{-a\mu}\right).
    \end{align*}
    Finally, using inequality~\eqref{eqn:finale} we get that
    \begin{align*}
        &\mathbb{E}[\|\mathbf{x}^{t}-\boldsymbol{\theta}^\star\|^2]\\&\leq 2 \left( \frac{128a^2{c_1^\prime}^2(MN-1)\left( L\Tilde{c}+\|\nabla F(\boldsymbol{\theta}^\star)\|^2 \right)}{((1-
    \delta)\varphi'(0)G)^2}\right.\\&\left.+O\left(t^{-\frac{\varphi'(0)G(1-\delta)}{4c_1^\prime}}\right)\right)\\
    &+ 2N\left(\frac{512a^2{c_1^\prime}^2L(MN-1)\left( L\Tilde{c}+\|\nabla F(\boldsymbol{\theta}^\star)\|^2 \right)}{N\mu((1-\delta)\varphi'(0)G)^2}\right.\\&\left.+ O\left(t^{-a\mu}\right)\right)\\
    &\leq \frac{256a^2{c_1^\prime}^2(MN-1)\left( L\Tilde{c}+\|\nabla F(\boldsymbol{\theta}^\star)\|^2 \right)}{((1-
    \delta)\varphi'(0)G)^2}\left(1+4\frac{L}{\mu}\right) \\&+O\left( t^{-\gamma}\right),
    \end{align*}
    for $\gamma=\min(\frac{\varphi'(0)G(1-\delta)}{4c_1^\prime}, a\mu),$ which completes the proof.
    
    \end{proof}
	
	\section{Numerical experiments}\label{section:examples}
    In this section we provide numerical examples that illustrate our main results in Section~\ref{section-convergenceanalysis}, and we also compare the proposed algorithm~\eqref{eq:alg1} with existing methods~\cite{nediccomnoise,Ourwork,Ourwork1}.\\

    \begin{example}
     We first analize the performance of~\eqref{eq:alg1} with respect to the step size constant $a.$ We consider a network where each agent $i=0,1,...,9$ has a scalar quadratic local objective, $f_i(x)=\frac{1}{2}(x-x_i^\star)^2,$ where $x_{i_k}^\star=i_k 1000-4500,$ for some permutation $i_k, k=1,2,...,10.$  Thus, the global minimizer is $x^\star=0.$ We set that $x^0_i=x_{i}^\star+5000$ for all agents $i=0,1,...,10.$ We set the algorithm parameters to $b=1, \delta=0.51$, and $a\in\{0.01,0.05,0.1,0.5,1\}.$
    The communication noise the following probability density function 
    \begin{align}\label{eq:htdistribution}
		f(w)=\frac{\beta-1}{2\,(1+|w|)^\beta}, 
	\end{align}
	with $\beta=2.05.$
    The underlying graph for the network is a randomly generated undirected graph with $10$ nodes.
    Figures~\ref{fig:acomparison50} and~\ref{fig:acomparison300} show a Monte Carlo-estimated MSE error for the different choices of~$a.$ We can see that there is a tradeoff between the asymptotic error and the convergence rate with respect to the choice of $a$: for $a=1$, the algorithm is the fastest in achieving its limiting error, and for $a=0.01,$ the algorithm has the lowest limiting error, but it is the slowest in achieving it. 
    This is in complete accordance with Theorem~\ref{theorem-MSE}.

\begin{figure}
    \centering
    \subfigure[]
    {
        \includegraphics[height=60mm]{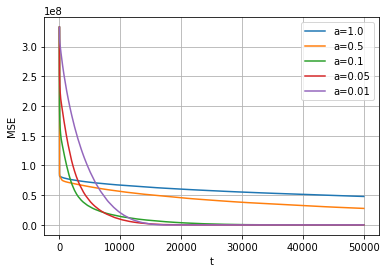}
        \label{fig:acomparison50}
    }
    \\
    \subfigure[]
    {
        \includegraphics[height=60mm]{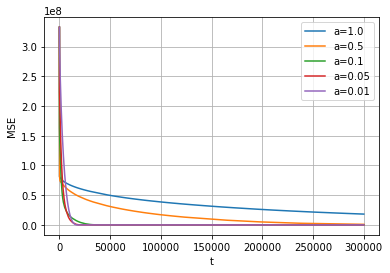}
        \label{fig:acomparison300}
    }
    \subfigure[]
    {
        \includegraphics[height=60mm]{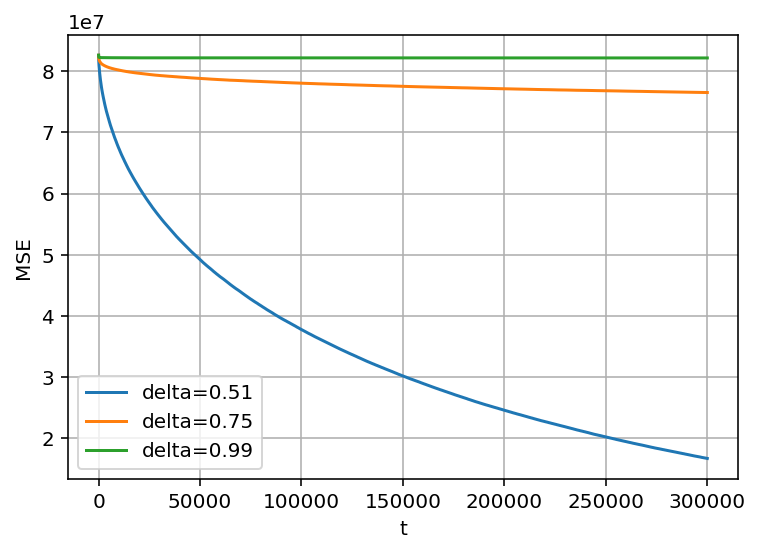}
        \label{fig:deltacomparison}
    }
    \caption{(a) Monte Carlo-estimated MSE error of proposed algorithm for the different choices of~$a$ for $50000$ iterations
        (b) Monte Carlo-estimated MSE error of proposed algorithm for the different choices of~$a$ for $300000$ iterations
        (c) Monte Carlo-estimated MSE error of proposed algorithm for the different choices of~$\delta$ }
\end{figure}
    \end{example}

    \begin{example}
     In this example, we analyze how the MSE changes by changing parameter $\delta$ in the definition of $\beta_t$.  The model setting is the same as in Example 1, except that now the algorithm parameters are $a=1,b=1$, and $\delta\in\{0.51,0.75,0.99\},$ and $x^0_i=x_{i}^\star+100$ for all agents $i=0,1,...,10.$ From  Figure~\ref{fig:deltacomparison}, it can be seen that the proposed algorithm has the best performance when $\delta=0.51.$ This is in accordance with Theorem~\ref{theorem-MSE}, where we recall that $\limsup\limits_{t\to\infty}\mathbb{E}[\|\mathbf{x}^{t}-\boldsymbol{\theta}^\star\|^2]=O(\frac{1}{1-\delta}).$
     \end{example}

     \begin{example} 
     In this example we compare the proposed algorithm~\eqref{eq:alg1} with a naive extension of the algorithm proposed in~\cite{Ourwork,Ourwork1} to the distributed optimization setting considered here. More precisely, we compare~\eqref{eq:alg1} with:
     \begin{align}\label{eqn:algext}
		\mathbf{x}_{i}^{t+1}=\mathbf{x}_{i}^{t}-\beta_{t}\sum_{j\in\Omega_{i}}
		\boldsymbol{\Psi}\left( \mathbf{x}_{i}^{t}-\mathbf{x}_{j}^{t} 
		+\boldsymbol{\xi}_{ij}^t\right)-\beta_t\nabla f_i(\mathbf{x}_i).
	\end{align}
     We consider the same model setting as in Example 2. It can be seen in Figure~\ref{fig:extensioncomparison} that the proposed algorithm outperforms algorithm~\eqref{eqn:algext}, where algorithm~\eqref{eqn:algext} exhibits a very large MSE. 
     

    \begin{figure}
    \centering
    \subfigure[]
    {
        \includegraphics[height=60mm]{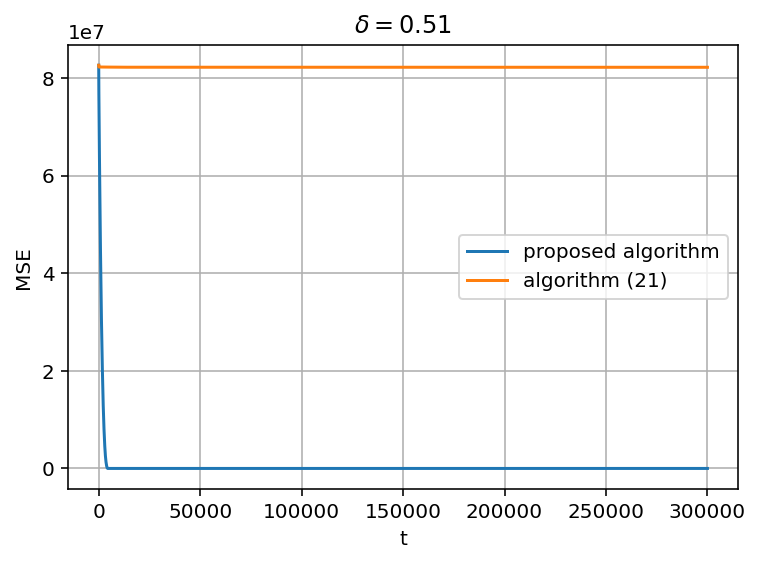}
    }
    \\
    \subfigure[]
    {
        \includegraphics[height=60mm]{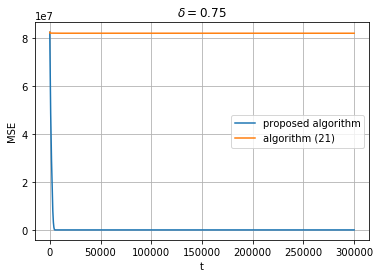}
    }
    \subfigure[]
    {
        \includegraphics[height=60mm]{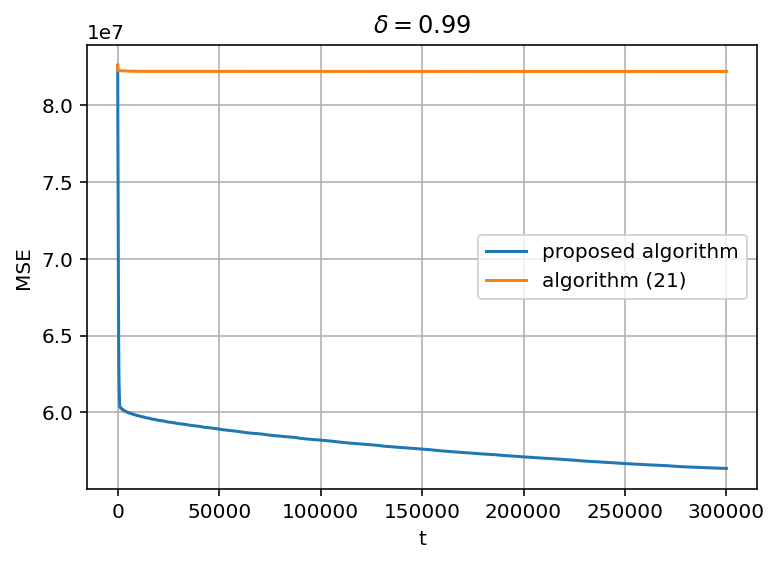}
    }
    \caption{Monte Carlo-estimated MSE error for proposed algorithm and for algorithm~\eqref{eqn:algext} for (a) $\delta=0.51,$ (b)  $\delta=0.75,$ (c) $\delta=0.99.$}
        \label{fig:extensioncomparison}
\end{figure}
\end{example}
    
    The next two examples compares the proposed algorithm with the one from~\cite{nediccomnoise} on the real data.

     \begin{example} 
     In this example we compare the proposed algorithm with the algorithm form~\cite{nediccomnoise} on a real data set and a standard supervized machine learning problem. In more detail, we  consider the Mushrooms dataset~[https://www.csie.ntu.edu.tw/~cjlin/libsvmtools/datasets/binary.html] (the original data set is slightly sub-sampled such that each agent observes the same number of data points), and we use the logistic loss function with $L2$ regulizer. Each agent has $100$ data points  and its local objective function is 
     \begin{align}\label{eqn:localobjective}
         f_i(\mathbf{x})=\sum\limits_{k=1}^{100} \ln( 1 + \exp(-b_{i_k}(a^\top_{i_k} \mathbf{x}))) +  \kappa \|\mathbf{x}\|^2,
     \end{align}
     for $\kappa=0.5.$ 
    Set is divided into a train set (6400 data points) and test set (1600 data points) and the underlying network is given with a randomly generated undirected graph with $64$ nodes.
	We compare the algorithms by using the global average value of agents' estimates - $\overline{\mathbf{x}}^t$ with respect to the CAP (Cumulative Accuracy Profile) curve and through the evaluation of test accuracy on a test set over iterations. Parameters for both methods are set to be gradient step size constant $a=1$ , consensus step size constant $b=1,$ consensus step size rate $\delta=0.51,$ and gradient step size rate in algorithm~\cite{nediccomnoise} is set to be $1$ as it is in proposed algorithm. We set that each agent starts from the origin. As it can be seen in Figures~\ref{fig:vsnedicCAP} and~\ref{fig:vsnedicacc}, the proposed algorithm outperforms the algorithm from~\cite{nediccomnoise}. 
        \begin{figure}[H]
    \centering
    \subfigure[]
    {
        \includegraphics[height=60mm]{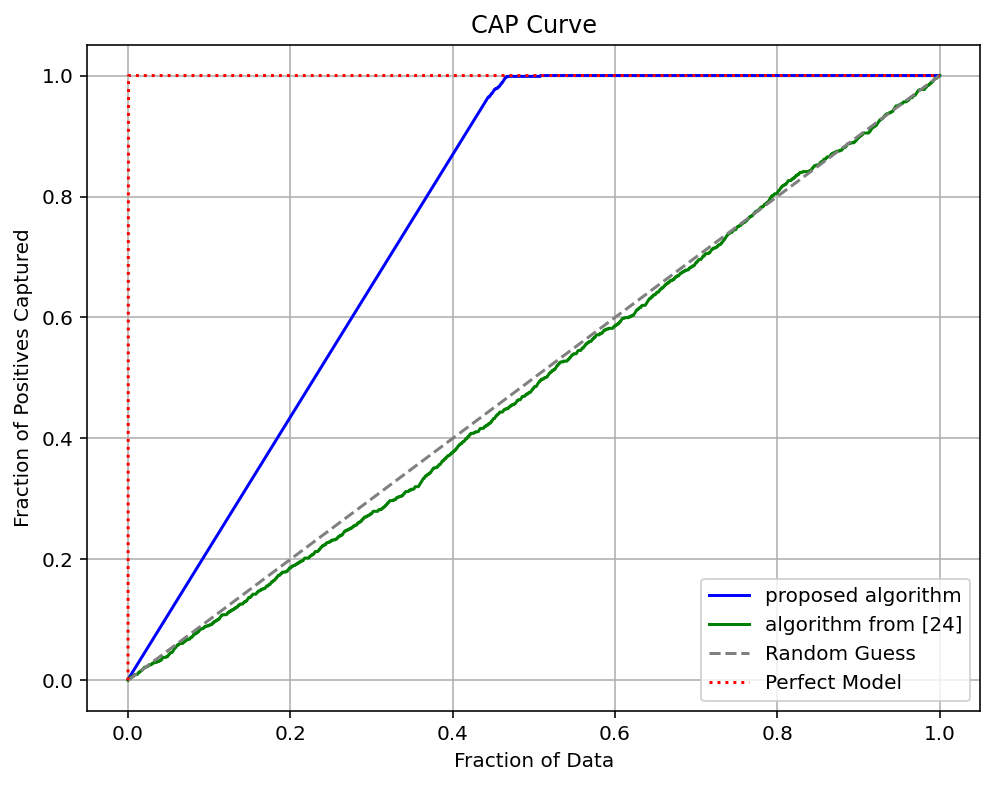}
        \label{fig:vsnedicCAP}
    }
    \\
    \subfigure[]
    {
        \includegraphics[height=60mm]{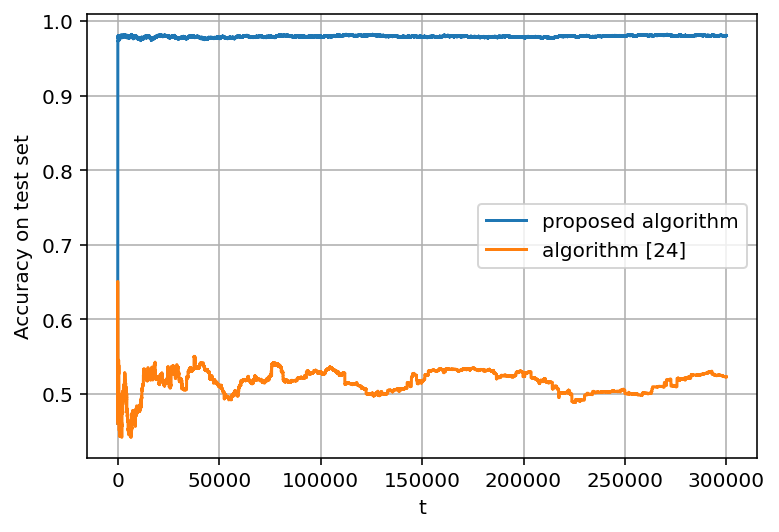}
        \label{fig:vsnedicacc}
    }
    \caption{
        (a) Monte Carlo-estimated CAP curve with proposed algorithm and algorithm form~\cite{nediccomnoise}
        (b) Monte Carlo-estimated accuracy through interation on test set for proposed algorithm and for algorithm from~\cite{nediccomnoise}}
\end{figure}
\end{example}

\begin{example} 
     In the last example we compare the proposed algorithm with the algorithm form~\cite{nediccomnoise} on a real data set and a standard supervized machine learning problem. Here, we condiser Adult data set a9a~[https://www.csie.ntu.edu.tw/~cjlin/libsvmtools/datasets/binary.html]
     (as it is in the previous Example, the original data set is slightly sub-sampled such that each agent observes the same number of data points), and we use the logistic loss function with $L2$ regulizer. Here, each agent has $500$ data points  and its local objective function is given by~\ref{eqn:localobjective}.
    Underlying network is given with a randomly generated undirected graph with $650$ nodes. Every parameter is the same as in the previous example.
    Here, we used the same methods as in previous Example to compare proposed algorithm~\eqref{eq:alg1} with the one from~\cite{nediccomnoise}. 
    Again, Figures~\ref{fig:vsnedicCAPad} and~\ref{fig:vsnedicaccad}, shows that the proposed algorithm outperforms the algorithm from~\cite{nediccomnoise}.  
        \begin{figure}[H]
    \centering
    \subfigure[]
    {
        \includegraphics[height=60mm]{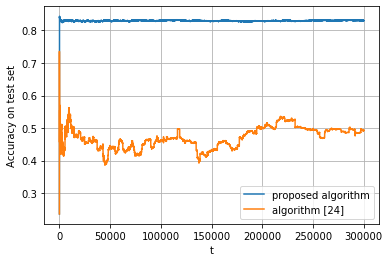}
        \label{fig:vsnedicCAPad}
    }
    \\
    \subfigure[]
    {
        \includegraphics[height=60mm]{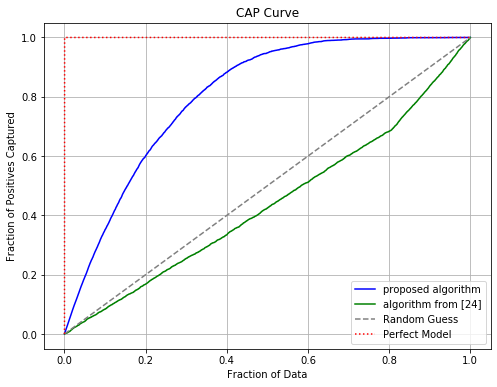}
        \label{fig:vsnedicaccad}
    }
    \caption{
        (a) Monte Carlo-estimated CAP curve with proposed algorithm and algorithm form~\cite{nediccomnoise}
        (b) Monte Carlo-estimated accuracy through interation on test set for proposed algorithm and for algorithm from~\cite{nediccomnoise}}
\end{figure}
\end{example}
    
	\section{Conclusion} \label{section-conlusion}
    We addressed the fundamental problem of designing and analyzing distributed optimization methods under heavy-tailed (infinite variance) communication noise. We developed a distributed gradient-type mixed time scale algorithm that is provably robust to heavy-tailed communication noise. 
     Specifically, the mean squared error (MSE) with the proposed method 
    and strongly convex heterogeneous local costs behaves as $O(a^2+1/t^a)$, 
    where $t$ is the iteration counter and $a$ is a constant that defines the consensus weights in the algorithm update rule. Numerical experiments corroborate the proposed method's resilience to heavy-tailed communication noise. At the same time, they show that existing distributed methods \cite{nediccomnoise} 
    designed to be robust to finite-variance communication noise 
    fail under the infinite-variance noise setting considered here.

    \bibliographystyle{siam}
	\bibliography{references.bib}
\end{document}